\documentclass[12pt, a4paper]{article}
\usepackage{amsmath, amsthm, amsfonts, amscd, amssymb, epsfig, setspace, url, graphics, MnSymbol}
\usepackage[all]{xy}

\newtheorem{theorem}{Theorem}[section]

\newtheorem{lemma}[theorem]{Lemma}
\newtheorem{corollary}[theorem]{Corollary}

\newtheorem{definition}[theorem]{Definition}

\newtheorem{remark}[theorem]{Remark}

\newcommand{\C}{{\mathbb C}}

\newcommand{\PP}{{\mathbb P}}

\newcommand{\SFF}{{II}} 

\DeclareMathOperator{\codim}{{codim}}

\DeclareMathOperator{\Sec}{{Sec}}

\DeclareMathOperator{\Hom}{{Hom}}

\DeclareMathOperator{\Singloc}{{Singloc}}
\DeclareMathOperator{\Ann}{{Ann}}
\DeclareMathOperator{\Baseloc}{{Baseloc}}

\title{A note on secant defective varieties and Clifford modules}
\author{Oliver Nash\\Imperial College London}
\date{June 2023}

\begin{document}
\maketitle

\begin{abstract}
  We generalise a construction of Landsberg, which associates certain Clifford
  algebra representations to Severi varieties. We thus obtain a new proof of
  Russo's Divisibility Property for LQEL varieties.
\end{abstract}

\section{Introduction}\label{Intro_sect}
The geometry of secant-defective varieties is surprisingly rich. In the early
20\textsuperscript{th} Century, the subject captured the attention of several
members of the Italian School of Algebraic Geometry and important results
appear in numerous beautiful old papers, such as those of Scorza \cite{Scorza},
Severi \cite{Severi}, and Terracini \cite{Terracini}.

In the 1980s the subject enjoyed a renaissance, largely due to a series of
breakthroughs made by Zak \cite{MR1234494}. Zak's solution of Hartshorne's
linear normality conjecture \cite{MR0384816} lead to his classification of
maximally-secant-deficient varieties, which he named Severi varieties. He
showed that there are exactly four Severi vareties, that they correspond to
normed division algebras, and that their dimensions are 2, 4, 8, 16. The
hardest part of the classification was establishing the dimension restriction.

In 1996, an intruiging paper of Landsberg \cite{MR1422359} appeared in which
he showed that the extrinsic geometry of a Severi variety induces certain
Clifford algebra representations. Using the classification of Clifford
modules, it is then trivial to see that the dimensions of the Severi varieties
must take the values already established by Zak.

The main purpose of this note is to show that Landsberg's Clifford modules
may be generalised. Severi varieties are examples of a class of varieties
introduced by Russo \cite{MR2501303} in 2009, called LQEL varieties and we show
that Landsberg's construction works in this more general setting.

Actually the generalisation is only an extremely mild extension of Landberg's
results. However we believe it is worth highlighting because, together with the
classification of Clifford modules, it provides a new proof of Russo's
Divisibility Property for LQEL varieties (corollary \ref{cor:lqel_divisibility}).

\section{Secant defective varieties and Clifford modules}
We shall follow Landsberg \cite{MR3586335, MR1422359} closely and so
recall his constructions and notation\footnote{Note that \cite{MR3586335} (which we follow)
adopts slightly different conventions than \cite{MR1422359}. For example
$\Ann(v)$ in \cite{MR1422359} corresponds to $\PP\Ann(v)$ in \cite{MR3586335}.}.

Our primary object of concern is a subvariety of projective space. We write:
\begin{align*}
  X \subseteq \PP^{n+a},
\end{align*}
to indicate that the variety $X$ is $n$-dimensional and that the embedding has
codimension $a$. We work over $\C$ throughout, and assume that $X$ is smooth,
irreducible, and non-degenerate\footnote{Not contained in a hyperplane.}, with
secant deficiency $\delta \ge 1$.

\subsection{Second fundamental form}
We recall \cite{MR559347} that the second fundamental form of an embedding
$X \subseteq \PP^{n+a}$ is a section of $\Hom(S^2TX, N)$, where $S^2TX$ is the
symmetric square of the tangent bundle and $N$ is the normal bundle. Thus for
$x \in X$, we have a symmetric bilinear map:
\begin{align*}
  \SFF^x : S^2T_xX \to N_x.
\end{align*}
When we have chosen a point $x \in X$ and there is no possibility of confusion,
we will write $T$ for $T_x X$, $N$ for $N_x$ and $\SFF$ for $\SFF^x$.
Taking the transpose, we also regard the second fundamental form as a linear
system of quadrics:
\begin{align*}
  \SFF^* : N^* \to S^2T^*.
\end{align*}

A key observation is that global properties of $X$ are visible infinitesimally
via the second fundamental form, and exceptional global properties tend to produce
linear systems of quadrics with exceptional properties.

For example, if $X$ has a smooth dual variety, then at a general point $\SFF^*$
is a linear system of quadrics of \emph{constant rank}, and from this follows
Landman's parity theorem for dual-deficient varieties
(see \cite[Theorem 12.4.8 and Corollary 12.4.10]{MR3586335}).

We shall show that if $X$ has the exceptional property that the
Gauss map of a general tangential projection has zero-dimensional fibres,
then its second fundamental form can be used to construct certain Clifford
modules, and from this follows Russo's Divisibility Property for
secant-deficient varieties (see \cite[Theorem 2.8]{MR2501303} or
\cite[Theorem 4.2.8]{MR3445582}).

\subsection{Tangential projections}
We now assume $X$ is non-linear\footnote{Note that this is automatic if
$\Sec(X) \ne \PP^{n+a}$ since $X$ is non-degenerate.} make two definitions to
fix notation:
\begin{definition}
  Let $x \in X$ and let $\mathbb{T}_x X \subseteq \PP^{n+a}$
  be the embedded tangent space at $x$. Away from $\mathbb{T}_x X$ we define a
  rational map:
  \begin{align*}
    \pi_x : X &\dashrightarrow \PP N,\\
    y &\mapsto [\langle y, \mathbb{T}_x X \rangle],
  \end{align*}
  where $N$ is the fibre of normal bundle of $X$ at $x$.
  The map $\pi_x$ is known as the tangential projection map at $x$.
\end{definition}
We recommend \cite[\S 2.3.2, \S3.3]{MR3445582} for a valuable discussion of
tangential projections.
\begin{definition}
  Let $x \in X$ and $\SFF$ be the second fundamental form at $x$.
  Away from $\Baseloc\SFF^*$ we define the rational map:
  \begin{align*}
    ii : \PP T &\dashrightarrow \PP N,\\
    [v] &\mapsto [\SFF(v, v)].
  \end{align*}
\end{definition}
We recall that the closures of the images of $\pi_x$ and $ii$ coincide and have
dimension $n - \delta$ (see e.g., \cite[Proposition 2.3.5]{MR3445582} and its
proof). Let this
$(n - \delta)$-dimensional irreducible subvariety be:
\begin{align*}
  Z \subseteq \PP N,
\end{align*}
and let:
\begin{align*}
  \Sec(X) \subseteq \PP^{n+a},
\end{align*}
be the $(2n + 1 - \delta)$-dimensional secant variety of $X$, then we note for
future reference that:
\begin{align}\label{eq:codim_sec_codim_Z}
  \codim Z = \codim \Sec(X).
\end{align}

\subsection{Second fundamental form of a tangential projection}
The following is essentially a restatement of \cite[Lemma 6.6]{MR1422359}.

Given a general point $x \in X$ and a general\footnote{Thus $v$ is non-zero and
$[v] \not\in \Baseloc \SFF^*$.} vector $v \in T$, let $z = ii([v])$. It follows
from the definition of $ii$ that there is a natural commutative diagram:
\begin{align*}
  \xymatrix{
    T \ar[r]^{\SFF_v} \ar[d]        & N \ar[d]\\
    T_{[v]}\PP T \ar[r]^{ii_{*[v]}} & T_z \PP N
  }
\end{align*}
where $ii_{*[v]}$ is the derivative of $ii$ at $[v]$ and $\SFF_v$ is the map:
\begin{align*}
  \SFF_v : w \mapsto \SFF(v, w).
\end{align*}
We thus have natural exact sequences:
\begin{align}\label{ses:rho_T}
  0 \to \langle v, \ker\SFF_v \rangle \to T \xrightarrow{\rho_T} T_z Z \to 0,
\end{align}
and:
\begin{align*}
  0 \to \SFF_v(T) \to N \xrightarrow{\rho_N} N_z Z \to 0,
\end{align*}
where $N_zZ$ is the normal bundle of $Z \subseteq \PP N$ at $z$.
The maps $\rho_T$ and $\rho_N$ fit into the following commutative diagram:
\begin{align}\label{sff_of_tproj}
  \begin{gathered}
    \xymatrix{
      S^2 T    \ar[r]^{\SFF} \ar[d]_{\rho_T\otimes\rho_T}   & N \ar[d]^{\rho_N}\\
      S^2 T_zZ \ar[r]^{\widetilde\SFF}                      & N_z Z
    }
  \end{gathered}
\end{align}
where $\widetilde\SFF$ is the second fundamental form of $Z$ at $z$.

\subsection{Singular locus of the second fundamental form}
Griffiths and Harris noticed that the quadrics of the second
fundamental form are all singular along the fibres of the Gauss map. In fact
they proved \cite[(2.6)]{MR559347}:
\begin{align*}
  \Singloc(N^*) = T_x F,
\end{align*}
where $F$ is the fibre of the Gauss map through $x$ and for any
$A \subseteq N^*$, $\Singloc(A)$ is the intersection of all the singular loci:
\begin{align}\label{singloc_eqn}
  \Singloc(A) = \bigcap_{f \in A} \{ v \in T ~|~ v \righthalfcup \SFF^*(f) = 0 \}.
\end{align}

A key insight of Landsberg \cite{MR1422359} was that there are natural
subsystems $A \subseteq N^*$ for which $\Singloc(A)$ captures more delicate
geometric features of $X$. Indeed it follows from \eqref{sff_of_tproj} that
there is a natural exact sequence:
\begin{align}\label{ses:singloc_ann}
  0 \to \ker\rho_T \to \Singloc(\rho_N^*(N^*_z Z)) \to \Singloc(N_z^*Z) \to 0,
\end{align}
and so for $v \in T$, Landsberg defined:
\begin{align*}
  \Ann(v) &= \rho_N^*(N^*_z Z)\\
          &= \SFF_v(T)^\perp\\
          &= \{ f \in N^* ~|~ v \righthalfcup \SFF^*(f) = 0 \},
\end{align*}
and studied the middle term $\Singloc(\Ann(v))$ appearing in \eqref{ses:singloc_ann}.

We note for future reference that:
\begin{align}\label{eq:dim_Ann}
  \dim \Ann(v) &= \codim Z\notag\\
               &= \codim \Sec(X) \quad\mbox{by \eqref{eq:codim_sec_codim_Z}}.
\end{align}

\subsection{Clifford modules}
The simplest class of secant-deficient varieties are those for which the Gauss
map of a general tangential projection has zero-dimensional fibres.
For emphasis we state a key consequence of this property:
\begin{lemma}\label{lem:minimal_singloc_ann}
  Let $X \subseteq \PP^{n+a}$ be a smooth, irreducible, non-degenerate variety
  of secant deficiency $\delta \ge 1$. Let $x \in X$, $v \in T_x X$ be general
  and let $Z \subseteq \PP N$ be the closure of the tangential projection at
  $x$. Then the Gauss map of $Z$ has zero-dimensional fibres if and only if:
  \begin{align}\label{eq:key_identity}
    \langle v, \ker \SFF_v \rangle = \Singloc(\Ann(v)).
  \end{align}
\end{lemma}
\begin{proof}
  Applying Griffiths and Harris's result \cite[(2.6)]{MR559347} to $Z$, we know
  that the Gauss map of $Z$ has zero-dimensional fibres if and only if the
  third term in \eqref{ses:singloc_ann} vanishes. Bearing in mind
  \eqref{ses:rho_T}, the conclusion is clear.
\end{proof}

\begin{remark}
  The Scorza Lemma \cite[Theorem 3.3.3]{MR3445582} tells us that the varieties
  satisfying the conditions of lemma \ref{lem:minimal_singloc_ann} are LQEL
  varieties. Moreover we have examples:
  \begin{itemize}
    \item the quadratic Veronese embeddings $\nu_2(\PP^n) \subseteq \PP^{n(n+3)/2}$ for $n \ge 2$ ($\delta = 1$),
    \item the binary Segre embeddings $\PP^n \times \PP^m \subseteq \PP^{mn + m + n}$ for $m + n \ge 3$ ($\delta = 2$),
    \item the rank-2 Pl\"ucker embeddings $G(2, n) \subseteq \PP^{(n-2)(n+1)/2}$ for $n \ge 5$ ($\delta = 4$),
    \item the 16-dimensional Severi variety in $\PP^{26}$ ($\delta = 8$),
  \end{itemize}
  as well as their linear projections. We recommend \cite{MR3445582} for
  further discussion.
\end{remark}

\begin{remark}
  In \cite[Definition 3.3.1]{MR3445582}, given general points $x, y \in X$
  and general $p \in \Sec(X)$ on the line $\overline{xy}$, Russo defines the
  contact locus $\Gamma_p \subseteq X$ as:
  \begin{align*}
    \Gamma_p = \overline{\{ x \in X ~|~ T_x X \subseteq T_p\Sec(X) \}},
  \end{align*}
  and notes that by Terracini's lemma:
  \begin{align*}
    \Sigma_p \subseteq \Gamma_p,
  \end{align*}
  where $\Sigma_p$ is the entry locus of $X$ with respect to $p \in \Sec(X)$.

  Let $Z = \overline{\pi_x(X)}$ be the closure of the tangential projection at
  $x$ and $F \subseteq Z$ be the fibre of the Gauss map of $Z$ through
  $z = \pi_x(y)$. As argued by Russo in the proof of
  \cite[Lemma 3.3.2]{MR3445582} the irreducible components of $\pi_x^{-1}(F)$
  and $\Gamma_p$ through $y$ coincide. We should thus have a natural exact
  sequence of tangent spaces:
  \begin{align}\label{ses:singloc_ann_alt}
    0 \to T_y \Sigma_p \to T_y\Gamma_p \to T_z F \to 0.
  \end{align}
  The line $\overline{yp}$ naturally determines a vector $v \in T_y X$, and we
  expect:
  \begin{align*}
    \Singloc(Ann(v)) = T_y\Gamma_p,
  \end{align*}
  so that \eqref{ses:singloc_ann_alt} can be interpreted as
  \eqref{ses:singloc_ann}. Given this, \cite[Lemma 3.3.2 (2)]{MR3445582} would
  provide an alternate proof of lemma \ref{lem:minimal_singloc_ann} above.
\end{remark}

We come at last to our main point:
\begin{theorem}\label{thm:clifford_module}
  Let $X \subseteq \PP^{n+a}$ be an smooth, irreducible, non-degenerate variety
  of secant deficiency $\delta \ge 1$ such that $\Sec(X) \ne \PP^{n+a}$.
  Suppose that the Gauss map of the tangential projection at a general point
  $x$ has zero-dimensional fibres, and let $v \in T_x X$ be general. Then
  $T / \Singloc(\Ann(v))$ carries a natural Clifford module structure over
  $\ker\SFF_v$.
\end{theorem}
\begin{proof}
  Let $Z = \overline{\pi_x(X)} \subseteq \PP N$ be the closure of the image of
  the tangential projection at $x$.
  
  The result we need is exactly
  \cite[Lemma 6.26]{MR1422359} except that we have made no assumption about
  $Z$ being a cone (instead assuming that its Gauss map has zero-dimensional
  fibres) and we do not assume that $Z$ is a
  hypersurface. In view of \eqref{eq:codim_sec_codim_Z}, $Z$ is a hypersurface
  if and only if $\Sec(X)$ is. However since the linear projection from a
  linear subspace which does not meet $\Sec(X)$ is an isomorphism, we may
  select a maximal such subspace and project down to the case that $\Sec(X)$
  is a hypersurface; the lemma then applies, and our proof is complete.

  For the benefit of readers who wish to compare with \cite{MR3586335}, we
  provide a reference for the argument as presented there. The key equation is
  \cite[(12.22) page 374]{MR3586335}:
  \begin{align*}
    q^{n+j}_{\epsilon\kappa}q^{n+k}_{\delta i} +
    q^{n+j}_{\delta k}q^{n+k}_{\epsilon i} =
    -2q^{n+1}_{\epsilon\delta}\delta^i_j \quad
    \forall \epsilon, \delta, j, k, i.
  \end{align*}
  The key assumption required for the derivation is $S = 0$ where:
  \begin{align*}
    S = \dim\Singloc\Ann(v) - \dim\langle v, \ker\SFF_v\rangle,
  \end{align*}
  which follows from lemma \ref{lem:minimal_singloc_ann}.
\end{proof}

\begin{remark}
  We can restate theorem \ref{thm:clifford_module} without referring
  to the second fundamental form as follows.

  Let $X$ be as in theorem \ref{thm:clifford_module} and let $p \in \Sec(X)$
  and $x \in X$ be general points. Let $Q \subseteq X$
  be the irreducible component of the $p$-entry locus through $x$
  and let $Q' \subseteq Q$ be the corresponding\footnote{See
  \cite[Lemma 2.4]{MR3228452}.} irreducible component of the tangent locus through $x$. Then
  $T_xQ'$ carries a non-degenerate quadratic form and the fibre $N^x_{Q|X}$
  of the normal bundle of $Q$ in $X$ is a Clifford module for the Clifford
  algebra $Cl(T_xQ')$.
\end{remark}

We emphasise the following corollary:
\begin{corollary}\label{cor:lqel_divisibility}
  Let $X$ be as in theorem \ref{thm:clifford_module} then:
  \begin{align*}
    \left. 2^{\left\lfloor\frac{\delta-1}{2}\right\rfloor} ~ \right\mid ~ n - \delta.
  \end{align*}
\end{corollary}
\begin{proof}
  The result follows immediately from theorem \ref{thm:clifford_module}
  together with the classification of Clifford modules. Indeed if
  there exists a $k$-dimensional Clifford module of a non-degenerate
  $l$-dimensional complex quadratic form, then:
  \begin{align*}
    \left. p ~ \right\mid ~ k,
  \end{align*}
  where $p = 2^{\left\lfloor\frac{l}{2}\right\rfloor}$.
  The reason is that the Clifford algebra of the quadratic form is the matrix
  algebra $\C^{p\times p}$ if $l$ is even or
  $\C^{p\times p} \oplus \C^{p\times p}$ if $l$ is odd
  (see e.g., \cite[Table 1]{MR0167985}) and the natural action of $\C^{p\times p}$
  on $\C^p$ is its unique irreducible representation.
\end{proof}

\begin{remark}
  The divisibility condition established in corollary
  \ref{cor:lqel_divisibility} was first proved by Russo and appeared in
  \cite[Theorem 2.8]{MR2501303} (see also \cite[Theorem 4.2.8]{MR3445582}). The
  proof involved an inductive study of the Hilbert scheme of lines through a
  general point of an LQEL variety.

  A second proof (due to the author) appeared as
  \cite[Corollary 2.6]{MR3228452}. This proof was topological and the key was
  a calculation in topological $K$-theory.

  We now have a third proof (albeit with slightly different assumptions) and
  this time it is Clifford module theory that is the key.

  It would be interesting to explore the relationship between the second and
  third proofs given the deep connections between $K$-theory and Clifford
  modules identified by Atiyah, Bott, and Shapiro in \cite{MR0167985}. The
  first step should be to argue that Landsberg's construction actually
  provides bundles of representations of Clifford algebras, as $x$ varies over
  a general tangent locus.
\end{remark}

\begin{remark}
  Note that the proof of corollary \ref{cor:lqel_divisibility} shows that the 2
  which appears in the expression $(\delta - 1)/2$ corresponds to the mod-2
  periodicity of Morita equivalence classes of complex Clifford algebras. Thus
  it is the same 2 which appears in complex Bott periodicity.
\end{remark}

\begin{remark}
  A quite different connection between Clifford modules and Severi varieties
  arises in the context of {\lq}Clifford structures{\rq}, introduced 
  by Moroianu and Semmelmann in \cite{MR2822214}. The Severi varieties appear
  in the classification of parallel non-flat even Clifford structures in
  \cite{MR2822214} (see also \cite{MR3383780}). It might be interesting to
  explore whether these Clifford structures have any relationship to
  Landsberg's Clifford modules.
\end{remark}

\section{A remark about the $\delta \le 8$ problem}
Let $X \subseteq \PP^{n+a}$ be a smooth, irreducible, non-degenerate, subvariety
with $\Sec(X) \ne \PP^{n+a}$.
It follows from Zak's proof of Hartshorne's linear normality conjecture that
the secant deficiency satisfies:
\begin{align}\label{lin_norm_secdef_bound}
  \delta \le \left\lfloor\frac{n}{2}\right\rfloor.
\end{align}
In the course of their exposition \cite{MR808175} of Zak's work, Lazarsfeld and
Van de Ven highlighted
that all known examples of $X$ as above satisfy $\delta \le 8$. They thus posed
the problem to investigate whether $\delta$ could be arbitrarily large (see
\cite[\S1f, page 19]{MR808175}). In view of \eqref{lin_norm_secdef_bound}, any
$X$ with $\delta > 8$ must have dimension $n \ge 18$.

Very little progress has been made on this problem in the 40 years since it was
first posed. Kaji \cite{MR1621761} has shown that any variety with $\delta > 8$
must be non-homogeneous but otherwise the problem remains completely
open: all known examples still satisfy $\delta \le 8$ and the 16-dimensional
Severi variety remains the only variety known to achieve $\delta = 8$. The
problem remains open even for the very special class of LQEL varieties (see
\cite[chapter 4.4, page 113]{MR3445582} as well as \cite[Remark 3.8]{MR2441252}
and \cite[Conjecture 4.5]{MR3294565}).

We mention this problem here, to highlight that by combining known
results, one may sharpen \eqref{lin_norm_secdef_bound} slightly as follows:
\begin{lemma}\label{lem:delta_bound}
  Let $X \subseteq \PP^{n+a}$ be a smooth, irreducible, non-degenerate subvariety
  with $\Sec(X) \ne \PP^{n+a}$. Suppose $n \ge 17$, then:
  \begin{align*}
    \delta \le \left\lfloor\frac{n - 1}{2}\right\rfloor.
  \end{align*}
\end{lemma}
\begin{proof}
  For a general tangential projection of $X$, let $\tilde\gamma$ be  the
  dimension of the fibre of its Gauss map and $\tilde\xi$ its dual deficiency.

  Suppose first that $\tilde\gamma = 0$. We may assume $\delta \ge 1$
  (otherwise there is nothing to prove) so by Scorza's Lemma
  \cite[Theorem 3.3.3]{MR3445582} $X$ is an LQEL variety. By
  \cite[Corollary 4.4.11]{MR3445582}:
  \begin{align*}
    \delta \le \left\lfloor\frac{n + 8}{3}\right\rfloor
           \le \left\lfloor\frac{n - 1}{2}\right\rfloor
  \end{align*}
  since $n \ge 17$.

  It remains to consider the case $\tilde\gamma \ge 1$. By
   \cite[Theorem 5.4.1, Lemma 3.3.2]{MR3445582}:
  \begin{align*}
    \delta \le \frac{n - \tilde\xi}{2}
           \le \frac{n - \tilde\gamma}{2}
           \le \frac{n - 1}{2},
  \end{align*}
  as required\footnote{We note in passing that one could instead deal with the
  case $\tilde\gamma \ge 1$ using results of Landsberg. Indeed after projecting
  to ensure $\Sec(X)$ is a hypersurface, one could apply
  \cite[Corollary 7.3]{MR1422359} (equivalently, the inequality at the bottom of
  page 373 of \cite{MR3586335}).}.
\end{proof}
Note that lemma \ref{lem:delta_bound} increases the dimension restriction on
a variety with $\delta > 8$ slightly to $n \ge 19$.
One might hope to make further progress
by studying varieties for which $\tilde\gamma = 1$ and then arguing as in
lemma \ref{lem:delta_bound} but with three cases corresponding to whether
$\tilde\gamma$ is 0, 1, or at least 2.

\bibliography{paper}

\begin{thebibliography}{10}

\bibitem{MR0167985}
M.~F. Atiyah, R.~Bott, and A.~Shapiro.
\newblock Clifford modules.
\newblock {\em Topology}, 3(suppl. 1):3--38, 1964.

\bibitem{MR559347}
Phillip Griffiths and Joseph Harris.
\newblock Algebraic geometry and local differential geometry.
\newblock {\em Ann. Sci. \'Ecole Norm. Sup. (4)}, 12(3):355--452, 1979.

\bibitem{MR0384816}
Robin Hartshorne.
\newblock Varieties of small codimension in projective space.
\newblock {\em Bull. Amer. Math. Soc.}, 80:1017--1032, 1974.

\bibitem{MR2441252}
Paltin Ionescu and Francesco Russo.
\newblock Varieties with quadratic entry locus. {II}.
\newblock {\em Compos. Math.}, 144(4):949--962, 2008.

\bibitem{MR3294565}
Paltin Ionescu and Francesco Russo.
\newblock On dual defective manifolds.
\newblock {\em Math. Res. Lett.}, 21(5):1137--1154, 2014.

\bibitem{MR3586335}
Thomas~A. Ivey and Joseph~M. Landsberg.
\newblock {\em Cartan for beginners}, volume 175 of {\em Graduate Studies in
  Mathematics}.
\newblock American Mathematical Society, Providence, RI, 2016.
\newblock Differential geometry via moving frames and exterior differential
  systems, Second edition [of MR2003610].

\bibitem{MR1621761}
Hajime Kaji.
\newblock Homogeneous projective varieties with degenerate secants.
\newblock {\em Trans. Amer. Math. Soc.}, 351(2):533--545, 1999.

\bibitem{MR1422359}
J.~M. Landsberg.
\newblock On degenerate secant and tangential varieties and local differential
  geometry.
\newblock {\em Duke Math. J.}, 85(3):605--634, 1996.

\bibitem{MR808175}
R.~Lazarsfeld and A.~Van~de Ven.
\newblock {\em Topics in the geometry of projective space}, volume~4 of {\em
  DMV Seminar}.
\newblock Birkh\"auser Verlag, Basel, 1984.
\newblock Recent work of F. L. Zak, With an addendum by Zak.

\bibitem{MR2822214}
Andrei Moroianu and Uwe Semmelmann.
\newblock Clifford structure on {R}iemannian manifolds.
\newblock {\em Adv. Math.}, 228(2):940--967, 2011.

\bibitem{MR3228452}
Oliver Nash.
\newblock {$K$}-theory, {LQEL} manifolds and {S}everi varieties.
\newblock {\em Geom. Topol.}, 18(3):1245--1260, 2014.

\bibitem{MR3383780}
Maurizio Parton and Paolo Piccinni.
\newblock The even {C}lifford structure of the fourth {S}everi variety.
\newblock {\em Complex Manifolds}, 2(1):89--104, 2015.

\bibitem{MR2501303}
Francesco Russo.
\newblock Varieties with quadratic entry locus. {I}.
\newblock {\em Math. Ann.}, 344(3):597--617, 2009.

\bibitem{MR3445582}
Francesco Russo.
\newblock {\em On the geometry of some special projective varieties}, volume~18
  of {\em Lecture Notes of the Unione Matematica Italiana}.
\newblock Springer, Cham; Unione Matematica Italiana, Bologna, 2016.

\bibitem{Scorza}
Gaetano Scorza.
\newblock Sulle variet\'a a quattro dimensioni di sr ($r \ge 9$) i cui $s_4$
  tangenti si tagliano a due a due.
\newblock {\em Rend. Circ. Mat. Palermo}, 27:148--178, 1909.

\bibitem{Severi}
Francesco Severi.
\newblock Intorno ai punti doppi impropri di una superficie generale dello
  spazio a quattro dimensioni, e a' suoi punti tripli apparenti.
\newblock {\em Rendiconti del Circolo Matematico di Palermo}, 15(1):33--51,
  1901.

\bibitem{Terracini}
Alessandro Terracini.
\newblock Sulle $v_k$ per cui la variet\`a degli $s_h (h+1)$-seganti ha
  dimensione minore dell'ordinario.
\newblock {\em Rend. Circ. Mat. Palermo}, 31:392--396, 1911.

\bibitem{MR1234494}
F.~L. Zak.
\newblock {\em Tangents and secants of algebraic varieties}, volume 127 of {\em
  Translations of Mathematical Monographs}.
\newblock American Mathematical Society, Providence, RI, 1993.
\newblock Translated from the Russian manuscript by the author.

\end{thebibliography}
\bibliographystyle{plain}
\end{document}